\documentclass[journal]{IEEEtran}
\usepackage{xcolor}
\usepackage[fleqn]{amsmath}
\usepackage{newtxtext}
\usepackage{empheq}
\usepackage{amssymb}
\usepackage{algorithm,algorithmic}
\usepackage{subfigure,amsfonts,balance}
\usepackage[dvips]{graphicx}
\DeclareMathOperator{\tr}{\mathrm{tr}}

\DeclareMathOperator{\rank}{\mathrm{rank}}
\DeclareMathOperator{\K}{\mathcal{K}}
\DeclareMathOperator{\E}{\mathbb{E}}

\newtheorem{theorem}{Theorem}
\newtheorem{proof}{Proof}
\newtheorem{remark}{Remark}

\begin{document}
%
\title{On the Interference Alignment Designs \\for Secure Multiuser MIMO Systems}


\author{Kha~Hoang~Ha,
        Thanh~Tung~Vu,
        Trung~Quang~Duong,
        and~Nguyen-Son~Vo
\thanks{K. H. Ha and T. T. Vu are with Ho Chi Minh City University of Technology, Vietnam (e-mail: hhkha@hcmut.edu.vn,\,\,tungvu.telecom@gmail.com).}
\thanks{T. Q. Duong is with Queen's University Belfast, UK (e-mail: trung.q.duong@qub.ac.uk).}
\thanks{N.-S. Vo is with Duy Tan University, Vietnam (e-mail: vonguyenson@dtu.edu.vn).}
}

\maketitle
\begin{abstract}
\boldmath
In this paper, we propose two secure multiuser multiple-input multiple-output (MIMO)
 transmission approaches based on interference alignment (IA) in the presence of an eavesdropper. To deal with the information leakage to the eavesdropper as well as the interference signals from undesired transmitters (Txs) at desired receivers (Rxs), our approaches aim to design the transmit precoding and receive subspace matrices to minimize both the total inter-main-link interference and the wiretapped signals (WSs). The first proposed IA scheme focuses on aligning the WSs into proper subspaces while the second one imposes a new structure on the precoding matrices to force the WSs to zero. In each proposed IA scheme, the precoding matrices and the receive subspaces at the legitimate users are alternatively selected to minimize the cost function of a convex optimization problem for every iteration. We provide the feasible conditions and the proofs of convergence for both IA approaches. The simulation results indicate that our two IA approaches outperform the conventional IA algorithm in terms of the average secrecy sum rate.
\end{abstract}

\IEEEpeerreviewmaketitle
\section{Introduction}
\label{sec:introd}
Due to the broadcasting nature of wireless communications, the wireless data transmission among legitimate users is susceptible to be wiretapped by nearby eavesdroppers. Although cryptographic schemes at the network layer seem to reach the security demands, they might not be suitable for large networks, because of complex encryption and decryption algorithms, key distribution and management \cite{mukherjee14}.
An alternative approach called physical layer security (PLS) exploits the random characteristics of wireless channels to achieve perfect secrecy without requiring encryption keys.
As such, PLS has received considerable attention from the research community \cite{Barros11,Zhou13,Trap15}. Secure approaches for various contemporary systems have been investigated, including MIMO systems \cite{Zhang15,Zhang13,Oggier08,Mukherjee13}, relaying systems \cite{Chen15,Zhang12}, heterogeneous and large-scale networks \cite{Yang15,Pinto12,Zhou11,Lv15}.

In multiuser multiple-input-multiple-output (MIMO) systems, various signal processing techniques can be exploited to enhance secrecy performance \cite{Hanif14,Yang14,Lv16,Geraci12,Mukherjee09,Yang13}.
The authors in \cite{Hanif14,Yang14,Lv16,Geraci12} have generated the precoding matrices to directly maximize the secrecy sum rate (SSR) of the network while the authors in \cite{Yang13,Mukherjee09} have designed the artificial noise properly to degrade the received signals at the eavesdropper without causing additional noise on the intended receiver (Rx). The solutions in \cite{Yang13,Hanif14,Yang14,Geraci12,Mukherjee09} address the optimization problems of the SSR in the broadcast channels in which only one base station communicates with multiple users in the network. However, they do not study the problem of multiple transmitter-receiver (Tx-Rx) pair transmissions. It should be emphasized that the sum rate maximization of the MIMO interference channels with multiple user pairs is mathematically challenging even for the scenarios that do not take secure transmission into consideration.

Recently, the degree of freedom (DoF) has been recognized as an important performance metric for MIMO interference channels since it can characterize the system capacity in the high signal-to-noise region.
The achievable DoFs can be obtained by interference alignment (IA) \cite{cadambe08,Papa12}. IA techniques have been applied to various wireless networks such as cognitive radio, relay networks, and multicell wireless networks (\cite{du11,Chen14,Zhou13b,Gao14} and references therein) while only a limited amount of work has exploited IA for secure multiuser MIMO transmissions \cite{Koyl11,Sasaki12}. In particular, the authors in \cite{Koyl11} have proven that it is possible for each user to achieve a nonzero secure DoF by using an IA scheme. Reference \cite{Sasaki12} demonstrated that the secure transmission is also possible when the number of antennas at the legitimate Tx and Rx is greater than that of the eavesdropper. However, a detailed IA design for secure multisuer MIMO communication networks has not been explored in \cite{Sasaki12}. Alternatively, the authors in \cite{Ruan16,Lee14} have exploited cooperative jamming schemes with IA for secrecy communications.

The conventional IA technique has been well studied in MIMO systems but not in a security context  \cite{cadambe08,kumar10,peters09,Papa12}. The authors in \cite{cadambe08} proposed a scheme which makes the Rxs free from inter-main-link-interference (IMLI) by forcing IMLI into a reduced-dimensional subspace. However, closed-form solutions of the precoders with more than $3$ users appear to be critical. Hence, iterative techniques have been investigated in \cite{kumar10,peters09,Papa12}, in which the precoding matrices at the Txs and the interference subspace matrices at the Rxs are alternatively updated in each iteration. Compared with the scheme in \cite{cadambe08}, the iterative algorithm is more general and applicable for the interference channels with an arbitrary number of Tx-Rx pairs. Motivated by these works, we adopt the iterative approach to our IA designs.

There are two typical scenarios of secrecy communications in multiuser MIMO wireless networks \cite{Koyl11}. In the first scenario, the secure communication networks involve the passive or external eavesdroppers. In such networks, CSI of wiretap channels is generally not available. The second scenario is that some particular confidential information is not intended for particular users (namely the eavesdroppers) in the networks. In such a case, all users (including the eavesdroppers) belong to the same network and, therefore, all CSI can be reasonably assumed to be available for the secrecy strategy design \cite{Shafiee09,Khisti2010,Oggier2011,Hanif14,Yang14,Geraci12,Koyl11}. The present paper investigates the secrecy communication with the multiple legitimate users and one eavesdropper in the same network. Given the global CSI, the focus of the paper is to design the transmission strategies of the legitimate users to maximize the secrecy capacity.
In particular, we proposed two IA schemes to reduce the information leakage at the eavesdropper while keeping the legitimate Rxs free from IMLI. Our IA designs reveal significant differences with respect to the conventional IA design \cite{peters09}, as well as provide system design guidelines to improve the secrecy. The main contributions of the paper are summarized as follows:
\begin{enumerate}
  \item We propose the first IA scheme, referred to as a wiretapped signal leakage minimization (WSLM) method
  in which the legitimate users cooperatively seek an optimal transmission strategy to minimize both IMLI and the wiretapped signals in their receive subspaces. The transmit precoding matrices and the receive subspace matrices are alternatively selected to monotonically reduce the cost function.
  \item We propose the second IA scheme, namely a zero-forcing wiretapped signal (ZFWS) scheme in which the legitimate transmitters align their signals into the null space of the channel matrices associated with the eavesdropper while the IMLIs are minimized. By zero-forcing WSs at the eavesdropper, the precoders are constructed by a cascade of two precoding matrices in which one matrix is the null space of the wiretap channel while the other matrix is designed to satisfy the IA conditions for the legitimate links.
  \item We prove that our IA algorithms are converged. The numerical results show that the proposed algorithms are converged in less than $50$ iterations. The feasible conditions are also analyzed and showed that the feasible condition for the ZFWS IA scheme is more restrictive than that for the WSLM IA design.  The feasible conditions also imply that the secrecy of the proposed system can be higher when increasing the number of antennas at Txs and Rxs.
  \item Simulation results indicate that the two proposed IA designs outperform the conventional one in terms of the SSR. The ZFWS IA provides the better SSR improvement than the WSLM for the systems in which the feasible conditions of both IA methods are satisfied. However, if the feasible conditions of the ZFWS IA are not satisfied, the ZFWS IA is inferior to the WSLM scheme.
\end{enumerate}

The rest of this paper is organized as follows. Section \ref{sec:Model} introduces the multiuser MIMO system in the presence of an eavesdropper and reviews the conventional IA design. In Section \ref{Sec:IAMUMIMO}, our IA algorithms are proposed. The numerical results are provided in Section \ref{sec:Results}. Finally, Section \ref{sec:Conclusion} presents concluding remarks.

\emph{Notation}: $\pmb{X}$ and $\pmb{x}$ are respectively denoted as a complex matrix and a vector. $\pmb{X}^T$ and $\pmb{X}^H$ are the transposition and conjugate transposition of $\pmb{X}$, respectively. $\pmb{I}$ and $\pmb{0}$ are respectively identity and zero matrices with the appropriate dimensions. $\tr (.)$, $\rank (.)$ and $\E(.)$  are the trace, rank and expectation operators, respectively. $||\pmb{X}||_F$ is the Frobenius norm. $\pmb{x}\sim \mathcal{CN}(\bar{\pmb{x}},\pmb{R}_{\pmb{x}})$ means that $\pmb{x}$ is a complex Gaussian random vector with means $\bar{\pmb{x}}$ and covariance $\pmb{R}_{\pmb{x}}$. $\mathcal{X}^\bot$ is denoted as the orthogonal subspace of the subspace $\mathcal{X}$.

Our preliminary results have been reported in conference papers \cite{TungATC2015,TungCMT2015}. This paper presents a more complete version of the work.

\section{System Model and Conventional IA Design}
\label{sec:Model}

\subsection{System model of secure multiuser MIMO communication systems}

We consider a secure multiuser MIMO communication system as shown in Fig. \ref{fig:DownlinkMIMOIC}. The system consists of $K$ Tx-Rx pairs in the presence of an eavesdropper (namely, the $(K+1)$-th receiver) trying to overhear information from each Tx-Rx pair. The $k$-th Tx equipped with $M$ antennas sends $d$ data streams to the $k$-th Rx equipped with $N$ antennas, where $k \in \mathcal{K}=\{1,...,K\}$. We consider the flat-fading MIMO channels where the channel matrix from the $\ell$-th Tx to the $k$-th Rx is denoted as $\pmb{H}_{k,\ell} \in \mathbb{C}^{N \times M}$, and the channel matrix from the $\ell$-th Tx to the eavesdropper is denoted as $\pmb{H}_{K+1,\ell} \in \mathbb{C}^{N_e \times M}$, where $N_e$ is the number of antennas at the eavesdropper. The channel entries are independent and identically distributed (i.i.d.) complex Gaussian variables and their magnitudes follow the Rayleigh distribution. The channel is assumed to be block-fading which remains unchanged for a frame duration and varies independently for every frame. In general, it is difficult to obtain CSI of wiretap channels in the general secure communication networks involving the passive or external eavesdroppers. However, obtaining CSI for wiretap channels is possible for certain scenarios.
One scenario is that the eavesdropper is one of the legacy users in the network and tries to access unauthorized services \cite{Li11,Lv16}. Another scenario is that the eavesdropper is a subscribed user in the network but it is not intended for particular confidential information \cite{Oggier2011}. Therefore, we can assume that CSI of all main and wiretapped links is known at all the legitimate users. This assumption has been widely adopted in numerous related works; see \cite{Lv16,Koyl11,Fako11,Shafiee09,Khisti2010,Oggier2011,Hanif14,Yang14,Geraci12} and references therein.

\begin{figure}[htb!]
\centering
\includegraphics[keepaspectratio,width=2.8in]{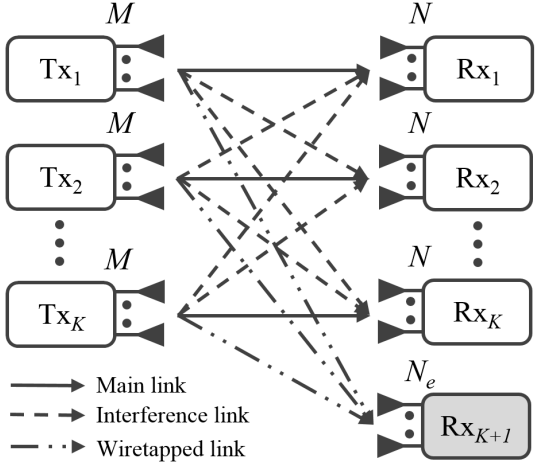}
\caption{A system model of secure multiuser MIMO communication networks.}
\label{fig:DownlinkMIMOIC}
\end{figure}
\vspace{+1cm}

In the multiuser MIMO system, the received signal $\pmb{y}_k \in \mathbb{C}^{N \times 1}$ at the $k$-th Rx can be expressed as
\begin{align}\label{Signal:Rx}
\pmb{y}_k = \sum\limits_{\ell\in\K} \pmb{H}_{k,\ell}\pmb{F}_{\ell}\pmb{x}_\ell  +\pmb{n}_k
\end{align}
where  $\pmb{x}_k \in \mathbb{C}^{d \times 1}$ is the signal vector transmitted from the $k$-th Tx to the $k$-th Rx, $\pmb{F}_k \in \mathbb{C}^{M \times d}$ is a precoding matrix at the $k$-th Tx and $\pmb{n}_k \sim \mathcal{CN}\left( 0,\sigma _k^2\pmb{I}_{N} \right)$ is a complex Gaussian noise vector at the $k$-th Rx. Without loss of generality, we assume that $\E\{ {{\pmb{x}_k}\pmb{x}_k^H} \} = {\pmb{I}_{{d}}}$ and, hence, the channel capacity at the $k$-th legitimate Rx, for all $k \in \mathcal{K}$, can be computed as \cite{Bazz12}
\begin{align}\label{Cap:Rxk}
\mathcal{R}_k = \log _2\left| \pmb{I}_{N} + \pmb{H}_{k,k}\pmb{F}_k\pmb{F}_k^H\pmb{H}_{k,k}^H\pmb{R}_{z_k}^{-1} \right|
\end{align}
where $\pmb{R}_{z_k} = \sum\limits_{\ell\in\K\setminus k} \pmb{H}_{k,\ell}{\pmb{F}_\ell\pmb{F}_\ell^H\pmb{H}_{k,\ell}^H}  + \sigma _k^2\pmb{I}_{N}$ is the interference plus noise correlation matrix at the $k$-th receiver. The eavesdropper (i.e., the $(K+1)$-th Rx) wiretaps the data signal from the $k$-th Tx-Rx pair and the rate of this wiretap channel can be given by
\begin{align}\label{Cap:WiretapRxk}
\mathcal{R}_k^{(e)} = \log _2\left| \pmb{I}_{N_e} + \pmb{H}_{K+1,k}\pmb{F}_k\pmb{F}_k^H\pmb{H}_{K+1,k}^H\pmb{R}_{e,k}^{-1} \right|
\end{align}
where $\pmb{R}_{e,k} = \sum\limits_{\ell \in \K \setminus k} \pmb{H}_{K+1,\ell}{\pmb{F}_\ell\pmb{F}_\ell^H\pmb{H}_{K+1,\ell}^H}  + \sigma _{K+1}^2\pmb{I}_{N_e}$ is the interference and noise correlation matrix with respect to the $k$-th desired signal at the eavesdropper. Now, the secrecy rate for the $k$-th Tx-Rx pair can be defined as \cite{Lv15,Lv16,Kong16}
\begin{align}\label{Cap:SRRxk}
\mathcal{R}_{S,k}  = [\mathcal{R}_k-\mathcal{R}_k^{(e)}]^{+},
\end{align}
where $[a]^{+}=\max(a,0)$. Then, the secure DoF (SDoF) of user $k$ can be defined as \cite{Lee14,Ruan16}
{\mathindent=0mm
\begin{equation*}
\text{SDoF}_k= \lim_{P_t \rightarrow \infty} \frac{\left[\mathcal{R}_k-\mathcal{R}^{(e)}_k \right]^+}{\log_2(P_t)}=[\text{dim}(\mathcal{W}_k)-\text{dim}(\mathcal{E}_k)]^+
\end{equation*}}
where $P_t$ is the transmitted power which is assumed to be the same at all the Txs; $\text{dim}(\mathcal{W}_k)$ is the dimension of the receive interference-free subspace at the legitimate receiver $k$ while $\text{dim}(\mathcal{E}_k)$ is the dimension of the receive interference-free subspace at the eavesdropper for wiretapping the signal from user $k$.
Accordingly, the SSR of the multiuser MIMO system can be expressed as \cite{Lv16,Geraci12}
\begin{align}\label{SSR}
\mathcal{R}_{S} = \sum\limits_{k\in\K}\mathcal{R}_{S,k} = \sum\limits_{k\in\K} [\mathcal{R}_k-\mathcal{R}_k^{(e)}]^{+},
\end{align}
which can be rewritten as
{\mathindent=0mm
\begin{equation}
\label{RsSumSDOF}
\mathcal{R}_S=(\sum_{k \in \mathcal{K}} [\text{dim}(\mathcal{W}_k)-\text{dim}(\mathcal{E}_k)]^+)\log_2(P_t)+o(\log_2(P_t)).
\end{equation}}
\noindent Eq.\eqref{RsSumSDOF} reveals the relationship between the SSR with the dimensions of the receive subspaces. Thus, the focus of our paper is to design the precoding and receive subspace matrices to increase the dimension of the desired signal subspace (by reducing the dimension of the interference subspace) at each legitimate user while simultaneously reducing the dimension of the wiretapped signal subspace.

\subsection{Conventional IA design for multiuser MIMO communication systems}
\label{IA:MUMIMO}
 In the conventional IA schemes for interference channels in \cite{peters09,Papa12}, there is no concern about the eavesdropper. The aim of the conventional IA design is to select the precoding $\left\{\pmb{F}_k\right\}_{k=1}^{K}$ and the receiving matrices $\left\{\pmb{W}_k\right\}_{k=1}^{K}$ to ensure no IMLI and to fully recover $d$ data streams at the intended Rx $k$. Mathematically, $\{\pmb{F}_k\}_{k=1}^K$ and $\{\pmb{W}_k\}_{k=1}^{K}$ satisfy the two following IA conditions
\begin{eqnarray}\label{ConventionalIA:SecrecyConditions1}
  \mathrm{rank}\big(\pmb{W}_k^H\pmb{H}_{k,k}\pmb{F}_k\big) &=& d  \\\label{ConventionalIA:SecrecyConditions2}
  \pmb{W}_k^H\pmb{H}_{k,\ell}\pmb{F}_\ell &=& \mathbf{0},\,\, \forall \ell \neq k, \ell \in \mathcal{K}.
\end{eqnarray}
Condition \eqref{ConventionalIA:SecrecyConditions1} guarantees that the intended signal at the $k$-th Rx achieves $d$ DoFs when $\pmb{H}_{k,k}$ is full rank. Condition \eqref{ConventionalIA:SecrecyConditions2} ensures no IMLI at the $k$-th Rx.
To remove the interference signals at the intended Rx $k$, the interference signal $\pmb{H}_{k,\ell}\pmb{F}_\ell$ is aligned into the interference subspace $\mathcal{U}_k$ which is spanned by $\pmb{U}_k$ and orthogonal with the receiving signal subspace $\mathcal{W}_k$ spanned by $\pmb{W}_k$. Mathematically, the concept of matrix distance between a matrix $\pmb{A}$ and an orthonormal matrix $\pmb{U}$ is defined as $||\pmb{A}- \pmb{U}\pmb{U}^H\pmb{A}||_\emph{F}^2$ \cite{peters09}.
Hence, the total interference leakage can be expressed as
\begin{align}\label{ConventionalJ}
\nonumber
 &\mathfrak{J}_{convt}(\{\pmb{F}_k\}_{k=1}^K, \{\pmb{U}_k\}_{k=1}^K) =
 \\
 &\sum\limits_{k=1}^{K}\sum\limits_{\ell=1,\ell \neq k}^{K}
 ||\pmb{H}_{k,\ell}\pmb{F}_\ell - \pmb{U}_k\pmb{U}_k^H\pmb{H}_{k,\ell}\pmb{F}_\ell ||_\emph{F}^2.
\end{align}
Finally, the IA problem can be mathematically posed as
\begin{subequations}\label{ConventionalOptPro}
\begin{eqnarray}
 \underset{\{\pmb{F}_k\}_{k=1}^K,\{\pmb{U}_k\}_{k=1}^{K}}{\text{min}} && \mathfrak{J}_{convt}(\{\pmb{F}_k\}_{k=1}^K, \{\pmb{U}_k\}_{k=1}^K)  \\
 \text{s.t.}\qquad\,\,\,\,  && \pmb{F}_k^H\pmb{F}_k= \frac{P_t}{d}\pmb{I}; \,\,\, \forall k \in \mathcal{K} \\
  && \pmb{U}_k^H\pmb{U}_k = \pmb{I}; \,\,\, \forall k \in \mathcal{K}.
\end{eqnarray}
\end{subequations}
The solution of the conventional IA problem in \eqref{ConventionalOptPro} is solved via an alternating algorithm as follows \cite{peters09}

\emph{Transmitter precoding design}:
By holding $\left\{\pmb{U}_k\right\}_{k=1}^{K}$ fixed, the total interference at the $k$-th Rx in \eqref{ConventionalJ} is minimized with the proper precoders $\left\{\pmb{F}_k\right\}_{k=1}^{K}$. The precoders $\left\{\pmb{F}_k\right\}_{k=1}^{K}$ are selected by solving the following optimization problem
{\mathindent=0mm
\begin{subequations}\label{OptProF}
\begin{align}
 \underset{\pmb{F}_\ell}{\mathrm{min}} && &\mathrm{tr}\{\pmb{F}_\ell^H(\sum\limits_{k\in\K \setminus \ell}\pmb{H}_{k,\ell}^H(\pmb{I}- \pmb{U}_k\pmb{U}_k^H)\pmb{H}_{k,\ell})\pmb{F}_\ell\}\\
 \text{s.t.} && &\pmb{F}_\ell^H\pmb{F}_\ell = \frac{P_t}{d}\pmb{I}.
\end{align}
\end{subequations}}
Then, the optimal solution $\pmb{F}_\ell$  is given by \cite{Lutke97}.
\begin{align}\label{Fl:IA0}
  \pmb{F}_\ell = \sqrt{\frac{P_t}{d}}\zeta_{\min}^d\{\sum\limits_{k\in\K \setminus \ell}\pmb{H}_{k,\ell}^H(\pmb{I}- \pmb{U}_k\pmb{U}_k^H)\pmb{H}_{k,\ell}\}
\end{align}
where $\zeta_{\min}^d\left\{\pmb{X}\right\}$ is the matrix whose columns are the $d$ eigenvectors corresponding to the $d$ smallest eigenvalues of $\pmb{X}$.

\emph{Receiver interference subspace selection}:
When $\pmb{F}_\ell$ is fixed, the orthonormal basis matrix $\pmb{U}_k$ at the legitimate Rx $k$ is selected to ensure that most of interference falls into the interference subspace $\mathcal{U}_k$ spanned by $\pmb{U}_k$. In terms of  the matrix variable $\pmb{U}_k$, the optimization problem \eqref{ConventionalOptPro} can be recasted as
\begin{subequations}\label{ConventionalOptProUk}
\begin{eqnarray}
 \underset{\pmb{U}_k}{\text{max}} && \mathrm{tr}\{\pmb{U}_k^H(\sum\limits_{\ell\in\K\setminus k}\pmb{H}_{k,\ell}\pmb{F}_\ell\pmb{F}_\ell^H\pmb{H}_{k,\ell}^H)\pmb{U}_k\} \\
 \text{s.t.} && \pmb{U}_k^H\pmb{U}_k = \pmb{I}.
\end{eqnarray}
\end{subequations}
Hence, $\pmb{U}_k$ can be given by \cite{Lutke97}
\begin{align}\label{ConventionalUk}
  \pmb{U}_k = \zeta_{\max}^{N-d}\{\sum\limits_{\ell\in\K\setminus k}\pmb{H}_{k,\ell}\pmb{F}_\ell\pmb{F}_\ell^H\pmb{H}_{k,\ell}^H\}
\end{align}
where $\zeta_{\max}^{N-d}\left\{\pmb{X}\right\}$ is the matrix whose columns are the $(N-d)$ dominant eigenvectors of matrix $\pmb{X}$.

It is worth noting that the conventional IA only aims at minimizing the total interference leakage which falls into a desired signal subspace in the insecure context. In the next section, we proposed two IA schemes for secure multiuser MIMO communications. The objective is not only to minimize the total interference leakage at the intended receivers but also to minimize the information bearing signal leakage at the eavesdropper. In addition, we derive the conditions on the number of users and antennas to guarantee the feasibility of secrecy communications.

\section{Interference Alignment for Secure Multiuser MIMO Communication Systems}
\label{Sec:IAMUMIMO}
In this section, we propose two IA schemes for secure multiuser MIMO communication systems. In addition to satisfying the conventional IA requirements for the legitimate users, the proposed schemes impose additional constraints to enhance the SSR. The first IA scheme, known as WSLM, aims to align all the information signals into a reduced dimensional subspace at the eavesdropper while the second one nulls out the signals at the eavesdropper.

\subsection{WSLM for secrecy multiuser MIMO communication systems}
\label{Sec:IAMUMIMO1}
To degrade the wiretap channel capacity, the underlaying idea is to align the signals from the legitimate users into a lower dimensional subspace at the eavesdropper. Similar to the conventional IA design, to remove the interference at the intended Rx $k$, we also align the interference signal $\pmb{H}_{k,\ell}\pmb{F}_\ell$ into the interference subspace $\mathcal{W}_k^{\perp}$ which is spanned by the orthonormal basis matrix $\pmb{U}_k$.
Hence, the total interference leakage inside the receiving subspace at the legitimate Rxs is defined by
\begin{align}\label{J1}
\nonumber
 &\mathfrak{J}_1(\{\pmb{F}_k\}_{k=1}^K, \{\pmb{U}_k\}_{k=1}^K)
 \\
 \nonumber
 &= \sum\limits_{k\in\K}\sum\limits_{\ell\in\K\setminus k}
 ||\pmb{H}_{k,\ell}\pmb{F}_\ell - \pmb{U}_k\pmb{U}_k^H\pmb{H}_{k,\ell}\pmb{F}_\ell ||_\emph{F}^2
 \\
 &=  \sum\limits_{k=\in\K}\sum\limits_{\ell\in\K\setminus k}
  \mathrm{tr}\left\{\pmb{F}_\ell^H\pmb{H}_{k,\ell}^H(\pmb{I}- \pmb{U}_k\pmb{U}_k^H)\pmb{H}_{k,\ell}\pmb{F}_\ell\right\}.
\end{align}
Note that the legitimate users do not have the knowledge of the wiretapped signal subspaces $\mathcal{E}_k$. To reduce the wiretapped information for all users, the legitimate users cooperatively seek a smallest dimensional subspace $\mathcal{E}$ at the eavesdropper in which the signals are most aligned. Assume that the subspace $\mathcal{E}$ at the eavesdropper is spanned by the orthonormal basis matrix $\pmb{U}_{K+1}$.  It should be emphasized that  $\mathcal{E}$ is appropriately selected by the legitimate transmitters. It is a subspace that all the signals can be most superimposed each other. The eavesdropper's receiver is not necessary to rely on $\mathcal{E}$. Mathematically, we aim to minimize
{\mathindent=0mm
\begin{align}
\nonumber
 &\mathfrak{J}_2(\{\pmb{F}_k\}_{k=1}^K, \pmb{U}_{K+1})
 \\
 &= \sum\limits_{k=1}^{K}||\pmb{H}_{K+1,\ell}\pmb{F}_\ell - \pmb{U}_{K+1}\pmb{U}_{K+1}^H\pmb{H}_{K+1,\ell}\pmb{F}_\ell ||_\emph{F}^2 \nonumber
 \\
 &=\sum\limits_{\ell=1}^{K}
  \mathrm{tr}\left\{\pmb{F}_\ell^H\pmb{H}_{K+1,\ell}^H(\pmb{I}- \pmb{U}_{K+1}\pmb{U}_{K+1}^H)\pmb{H}_{K+1,\ell}\pmb{F}_\ell\right\}. \label{J2:2}
\end{align}}
Then, the total signal leakage can be defined as
{\mathindent=1mm
\begin{align}
\nonumber
&\mathfrak{J}(\{\pmb{F}_k\}_{k=1}^K, \{\pmb{U}_k\}_{k=1}^{K+1}) =
\\
&\mathfrak{J}_1(\{\pmb{F}_k\}_{k=1}^K, \{\pmb{U}_k\}_{k=1}^K)+\mathfrak{J}_2(\{\pmb{F}_k\}_{k=1}^K, \pmb{U}_{K+1}).
\end{align}}
Accordingly, the IA problem can be expressed as the following joint optimization problem
{\mathindent=1mm
\begin{subequations}\label{OptProIA1}
\begin{align}\label{OptProIA1:CF}
\underset{\{\pmb{F}_k\}_{k=1}^K,\{\pmb{U}_k\}_{k=1}^{K+1}}{\text{min}} &\mathfrak{J}(\{\pmb{F}_k\}_{k=1}^K, \{\pmb{U}_k\}_{k=1}^{K+1})
\\
\text{s.t.} \qquad \,\,\,&\pmb{F}_k^H\pmb{F}_k = \frac{P_t}{d}\pmb{I}; \,\,\, \forall k \in \mathcal{K}
\\
&\pmb{U}_k^H\pmb{U}_k = \pmb{I};  \forall k \in \mathcal{K}\cup\{K+1\}.
\end{align}
\end{subequations}}
We now derive the solution of problem \eqref{OptProIA1} via alternating minimization \cite{peters09}. First, the variables $\{\pmb{U}_k\}_{k=1}^{K+1}$ are temporarily fixed, the objective function can be optimized with the variables $\{\pmb{F}_k\}_{k=1}^K$ and, then, we alternate between the two set of variables. Specifically, the key ideas in our proposed algorithm are explained in the following.

\emph{Transmitter precoding design}:
Firstly, the $k$-th Tx adjusts $\pmb{F}_k$ to guarantee the interference signal falling into the interference subspace $\mathcal{U}_k$ at the $k$-th Rx. When $\{\pmb{U}_k\}_{k=1}^{K+1}$ are fixed, for finding optimal $\pmb{F}_\ell$, problem \eqref{OptProIA1} equivalently reduces to
{\mathindent=0mm
\begin{subequations}\label{OptProF:IA1}
\begin{align}
 \underset{\pmb{F}_\ell}{\text{min}} && &\mathrm{tr}\{\pmb{F}_\ell^H(\sum\limits_{k=1,k \neq \ell}^{K+1}\pmb{H}_{k,\ell}^H(\pmb{I}- \pmb{U}_k\pmb{U}_k^H)\pmb{H}_{k,\ell})\pmb{F}_\ell\}
 \\
 \text{s.t.} && &\pmb{F}_\ell^H\pmb{F}_\ell = \frac{P_t}{d}\pmb{I}; \,\,\, \forall \ell \in \mathcal{K}.
\end{align}
\end{subequations}}
The solution to problem \eqref{OptProF:IA1} is given by \cite{Lutke97}
{\mathindent=1mm
\begin{align}\label{Fl:IA1}
  \pmb{F}_\ell = \sqrt{\frac{P_t}{d}}\zeta_{\min}^d\{\sum\limits_{k=1,k \neq \ell}^{K+1}\pmb{H}_{k,\ell}^H(\pmb{I}- \pmb{U}_k\pmb{U}_k^H)\pmb{H}_{k,\ell}\}.
\end{align}}

\emph{Receiver interference subspace selection}:
Secondly, when $\pmb{F}_\ell$ is fixed, similar to the conventional IA design, the orthonormal basis matrix $\pmb{U}_k$ at the legitimate Rx $k$ is selected for each Rx $k$ as in \eqref{ConventionalUk}. To deal with wiretapped signals, we define $\mathcal{E}^\bot$ spanned by an orthonormal matrix $\pmb{E} \in \mathcal{}^{N_e\times n}$ as an orthogonal subspace of the subspace $\mathcal{E}$. Since $\pmb{E}$ is an orthonormal basis matrix, $\mathrm{rank}\left(\pmb{E}\right)=\mathrm{rank}\left(\pmb{E}^H\right)=n$. The design of interest is to align the signals into the subspace $\mathcal{E}$ and, thus, we have the following conditions
\begin{equation}\label{IA:SecrecyConditions3}
\pmb{E}^H\pmb{H}_{K+1,\ell}\pmb{F}_\ell=\pmb{0}, \quad \ell \in \mathcal{K}.
\end{equation}
That is,  $\pmb{H}_{K+1,\ell}\pmb{F}_\ell \in \mathbb{C}^{N_e\times d}$ must lie in the null space of $\pmb{E}^H$. Hence, the existence condition of $\pmb{E}^H$ is $N_e-n\geq d$. Thus, we consider the case that $n=N_e-d$ for the dimension of the null space of $\pmb{E}^H$ to be smallest, to make the eavesdropper harder to recover signals \footnote{Note that different from the conventional IA scheme in which the receive subspace is tightly related to its corresponding receiver, the subspace $\mathcal{E}$ herein is selected by the legitimate users and is not related to the eavesdropper's receiver. The aim of the legitimate users is to find the subspace $\mathcal{E}$ such that their signals can be most aligned into the subspace $\mathcal{E}$ without regarding the eavesdropper's receiver.}. Then, the matrix size of $\pmb{U}_{K+1}$ must be $N_e\times d$. From \eqref{J2:2}, the orthonormal basis $\pmb{U}_{K+1}$ can be found by solving the following optimization problem
{\mathindent=1mm
\begin{subequations}\label{OptProUe}
\begin{align}
 \underset{\pmb{U}_{K+1}}{\text{max}} \,\,
 &\mathrm{tr}\{\pmb{U}_{K+1}^H(\sum\limits_{\ell\in\K}\pmb{H}_{K+1,\ell}\pmb{F}_\ell\pmb{F}_\ell^H\pmb{H}_{K+1,\ell}^H)\pmb{U}_{K+1}\} \\
 \text{s.t.} \,\,\,&\pmb{U}_{K+1}^H\pmb{U}_{K+1} = \pmb{I}_d.
\end{align}
\end{subequations}}
The solution to \eqref{OptProUe} can be found as follows \cite{Lutke97}
\begin{align}\label{Ue}
  \pmb{U}_{K+1} = \zeta_{\max}^d \{\sum\limits_{\ell\in\K}\pmb{H}_{K+1,\ell}\pmb{F}_\ell\pmb{F}_\ell^H\pmb{H}_{K+1,\ell}^H\}.
\end{align}
The step-by-step iterative algorithm for the WSLM method can be shown in Algorithm \ref{alg2}.

\begin{algorithm}[ht]
\caption{: Proposed  WSLM IA Algorithm for Secure Multiuser MIMO Systems}\label{IterIA}
\begin{algorithmic}[1]\label{alg2}
\STATE Inputs: $d,\pmb{H}_{k,\ell}$,  $\forall k \in \mathcal{K}\cup\{K+1\}$, $\forall \ell \in \mathcal{K}$, $\kappa=0$, $\kappa_{\max }$, where $\kappa$ is the iteration index;
\STATE Initial variables: random matrix $\{\pmb{F}_k^{(0)}\}_{k = 1}^{K}$ satisfied $\pmb{F}_k^{(0)H}\pmb{F}_k^{(0)}=\sqrt{\frac{P_t}{d}}\pmb{I}_d$; then select $\{\pmb{U}_k^{(0)}\}_{k = 1}^{K+1}$ from \eqref{ConventionalUk} and \eqref{Ue};
\STATE Evaluate the objective function $\mathfrak{J}(\{\pmb{F}_k^{(0)}\}_{k=1}^K, \{\pmb{U}_k^{(0)}\}_{k=1}^{K+1})$ from \eqref{OptProIA1};
\WHILE{$\kappa < \kappa_{\max}$}
\STATE For fixed $\{\pmb{U}_k^{(\kappa)}\}_{k = 1}^{K+1}$, select $\{\pmb{F}_k^{(\kappa+1)}\}_{k = 1}^K$ from \eqref{Fl:IA1};
\STATE For fixed $\{\pmb{F}_k^{(\kappa+1)}\}_{k = 1}^K$, select $\{\pmb{U}_k^{(\kappa+1)}\}_{k = 1}^{K+1}$ from \eqref{ConventionalUk} and \eqref{Ue};
\STATE Evaluate the objective function $\mathfrak{J}(\{\pmb{F}_k^{(\kappa+1)}\}_{k=1}^K, \{\pmb{U}_k^{(\kappa+1)}\}_{k=1}^{K+1})$ from \eqref{OptProIA1};
\STATE $\kappa=\kappa+1$;
\STATE Repeat steps 5-8 until convergence.
\ENDWHILE
\end{algorithmic}
\end{algorithm}

\emph{Convergence analysis for the WSLM IA design}:
The convergence of the proposed WSLM IA scheme can be stated in the following theorem.
\begin{theorem}\label{thrmalg1}
The proposed WSLM IA in Algorithm \ref{alg2} for secure multiuser MIMO communication networks converges monotonically.
\end{theorem}
\begin{proof}
In the $\kappa$-th iteration, by holding $\{\pmb{U}_k^{(\kappa)}\}_{k=1}^{K+1}$ fixed, in step $5$, the optimal solution $\{\pmb{F}_k^{(\kappa+1)}\}_{k=1}^K$ obtained from \eqref{Fl:IA1} minimizes $\mathfrak{J}(\{\pmb{F}_k\}_{k=1}^K, \{\pmb{U}_k^{(\kappa)}\}_{k=1}^{K+1})$ in \eqref{OptProIA1} and, thus,
\begin{equation}\label{conve1}
\mathfrak{J}(\{\pmb{F}_k^{(\kappa+1)}, \{\pmb{U}_k^{(\kappa)}\}) \leq \mathfrak{J}(\{\pmb{F}_k^{(\kappa)}\}, \{\pmb{U}_k^{(\kappa)}\}).
\end{equation}
 Next, in step $6$, by holding $\{\pmb{F}_k^{(\kappa+1)}\}_{k=1}^K$ fixed, the optimal solution $\{\pmb{U}_k^{(\kappa+1)}\}_{k=1}^{K+1}$ obtained from \eqref{ConventionalUk} minimizes $\mathfrak{J}(\{\pmb{F}_k^{(\kappa+1)}\}_{k=1}^K, \{\pmb{U}_k\}_{k=1}^{K+1})$, i.e.,
\begin{align}\label{conve2}
\mathfrak{J}(\{\pmb{F}_k^{(\kappa+1)}\}, \{\pmb{U}_k^{(\kappa+1)}\})
\leq \mathfrak{J}(\{\pmb{F}_k^{(\kappa+1)}\}, \{\pmb{U}_k^{(\kappa)}\}).
\end{align}
The combination of Eqs. \eqref{conve1} and \eqref{conve2} yields
 \begin{equation}
\mathfrak{J}(\{\pmb{F}_k^{(\kappa+1)}\}, \{\pmb{U}_k^{(\kappa+1)}\}) \leq \mathfrak{J}(\{\pmb{F}_k^{(\kappa)}\}, \{\pmb{U}_k^{(\kappa)}\}).
\end{equation}
That is, the cost function $\mathfrak{J}$ in \eqref{OptProIA1} is reduced monotonically over iteration. In addtion, the cost function is bounded below by zero. Thus, the convergence of the first proposed IA algorithm is guaranteed.
\end{proof}

\emph{Feasible condition for the proposed WSLM IA design}:
The feasibility of the IA scheme is tightly related to the properness of the system \cite{Yetis10}. The MIMO interference channels are known as proper systems as the number of variables $N_v$ is greater than or equal to the number of equations $N_{eq}$ in the IA conditions \eqref{ConventionalIA:SecrecyConditions2} and \eqref{IA:SecrecyConditions3}. The proper systems surely render the feasibility of the IA problems. To count $N_{eq}$ and $N_v$, we rewrite Eq. \eqref{ConventionalIA:SecrecyConditions2} and \eqref{IA:SecrecyConditions3} as follows
{\mathindent=1mm
\begin{eqnarray}\label{IA:SecrecyConditions22}
  \pmb{w}_k^{[i]H}\pmb{H}_{k,\ell}\pmb{f}_\ell^{[j]} = \mathbf{0};\,\, \forall \ell \neq k, \ell \in \mathcal{K},\forall i \in \mathcal{I},\forall j \in \mathcal{K} \\\label{IA:SecrecyConditions33}
  \pmb{e}^{[m]H}\pmb{H}_{K+1,\ell}\pmb{f}_\ell^{[j]} = \mathbf{0}; \,\, \forall \ell \in \mathcal{K},\forall m \in \mathcal{M},\forall j \in \mathcal{K}
\end{eqnarray}}
where $\pmb{f}_\ell^{[j]}$, and $\pmb{w}_k^{[i]}$, $\pmb{e}^{[m]}$ is the transmit beamforming vectors and column vectors of basis matrices at the Tx, Rx, and eavesdropper, respectively, with $\mathcal{I}=\{1,\dots,d\}$ and $\mathcal{M}=\{1,\dots,N_e-d\}$.
In the proposed IA design, $N_{eq}$ is directly given from \eqref{IA:SecrecyConditions22} and \eqref{IA:SecrecyConditions33} as follows
\begin{align}\label{Neq}
  N_{eq} &= K(K-1)d^2+K(N_e-d)d.
\end{align}

The number of variables designed for any precoder or subspace matrix $\pmb{V}\in \mathbb{C}^{m \times n}$ is proven to be equal to $n(m-n)$ \cite{Yetis10}. Hence, the number of variables to be designed for precoder $\pmb{F}_\ell$, subspace matrices $\pmb{W}_k$ and $\pmb{E}$ equals $d(M-d)$, $d(N-d)$ and $d(N_e-d)$, respectively. Finally, the total number of variables in the proposed IA conditions \eqref{IA:SecrecyConditions22} and \eqref{IA:SecrecyConditions33} can be given as
\begin{align}\label{Nv}
  N_v = Kd\left(M+N-2d\right)+d(N_e-d).
\end{align}
Denote $\left(M\times N,N_e,d\right)^K$ as the system where there are $K$ Tx-Rx pairs and one eavesdropper; the Tx, Rx and eavesdropper are respectively equipped with $M$, $N$, $N_e$ antennas; and $d$ data streams are sent between each Tx-Rx pair.
The feasible condition for the proposed IA design in secure communications can be stated in Theorem \ref{thrmalg2}.

\begin{theorem}\label{thrmalg2}
The IA scheme in Algorithm \ref{alg2} for the $\left(M\times N,N_e,d\right)^K$ system with secure communication is feasible if $N_v\geq N_{eq}$, i.e., \begin{align}\label{feasibleIA1condition}
  K(M+N)-\left(K^2+1\right)d \geq N_e(K-1).
\end{align}
\end{theorem}
\begin{proof}
  By comparing $N_{eq}$ in \eqref{Neq}  and $N_v$ in \eqref{Nv}, one can easily obtain \eqref{feasibleIA1condition}.
\end{proof}
\begin{remark}
Theorem \ref{thrmalg2} implies that for satisfying the WSLM IA feasible conditions, the number of antennas at the eavesdropper is restricted by
$N_e\leq\frac{K(M+N)-\left(K^2+1\right)d}{K-1}$. Since the number of antennas at the eavesdropper cannot be controlled in the proposed system, the secrecy can be enhanced when increasing the number of transceiver antennas $M$ and $N$.
\end{remark}
\subsection{ZFWS based IA design for secure multiuser MIMO systems}
In this subsection, we propose an IA design to force the information leakage to zero at the eavesdropper while keeping two conventional IA conditions in \eqref{ConventionalIA:SecrecyConditions1} and \eqref{ConventionalIA:SecrecyConditions2} satisfied. Condition \eqref{IA:SecrecyConditions3} can be rewritten as a zero-information-leakage constraint
\begin{align}\label{IA:zero-leakage-constrain}
\pmb{H}_{K+1,\ell}\pmb{F}_\ell = \mathbf{0}, \quad \forall \ell \in \mathcal{K}.
\end{align}
From Eq. \eqref{IA:zero-leakage-constrain}, we use a zero-forcing method to cancel all received signals at the eavesdropper. The key idea is to find the precoding matrices $\pmb{F}_\ell$  to force all the receiving signals to zero at the eavesdropper. It means that $\pmb{F}_\ell$ is the null space of $\pmb{H}_{K+1,\ell}$. To this end, we calculate the singular value decomposition (SVD) of $\pmb{H}_{K+1,\ell}$ as
\begin{align} \label{HE:SVD}
\pmb{H}_{K+1,\ell} = \pmb{\Phi}_\ell\pmb{\Sigma}_\ell\pmb{\Xi}_\ell^H
\end{align}
where $\pmb{\Sigma}_\ell$ is the diagonal matrix with diagonal elements being singular values in decreasing order. For the existence of $\pmb{F}_\ell$ satisfying Eq. \eqref{IA:zero-leakage-constrain}, we must have $M-N_e \geq d$. Define $\pmb{\Delta}_\ell \in \mathbb{C}^{M \times d}$ as the matrix of the last $d$ columns in the matrix $\pmb{\Xi}_\ell$, i.e., $\pmb{\Delta}_\ell$ lies on the null space of $\pmb{H}_{K+1,\ell}$. Hence, the precoder matrix $\pmb{F}_\ell$ is defined as $\pmb{F}_\ell=\pmb{\Delta}_\ell\pmb{P}_\ell$, where $\pmb{P}_\ell \in \mathbb{C}^{d\times d}$ is an arbitrary matrix. The total interference leakage $\mathfrak{J}$ in \eqref{ConventionalJ} can be rewritten as
\begin{align}\label{J:IA2}
\nonumber
 &\mathfrak{J}(\{\pmb{P}_k\}_{k=1}^K, \{\pmb{U}_k\}_{k=1}^{K}) =
 \\
 &\sum\limits_{k\in\K}\sum\limits_{\ell\in\K\setminus k}
 ||\pmb{H}_{k,\ell}\pmb{\Delta}_\ell\pmb{P}_\ell - \pmb{U}_k\pmb{U}_k^H\pmb{H}_{k,\ell}\pmb{\Delta}_\ell\pmb{P}_\ell ||_\emph{F}^2.
\end{align}
Then, the IA problem can be expressed as an optimization problem
\begin{subequations}\label{OptProIA2}
\begin{align}
 \underset{\{\pmb{P}_k\}_{k=1}^K,\{\pmb{U}_k\}_{k=1}^K}{\text{min}}  & \mathfrak{J}(\{\pmb{P}_k\}_{k=1}^K, \{\pmb{U}_k\}_{k=1}^{K})
 \\
 \text{s.t.} \qquad\,\,\, & \pmb{P}_k^H\pmb{P}_k=  \frac{P_t}{d}\pmb{I} \,\,\, \forall k \in \mathcal{K} \\
  & \pmb{U}_k^H\pmb{U}_k = \pmb{I}; \,\,\, \forall k \in \mathcal{K}.
\end{align}
\end{subequations}
We now derive the solution of problem \eqref{OptProIA2} via alternating minimization as follows \cite{peters09}.

\emph{Transmitter precoding design}:
The interference leakage power in \eqref{J:IA2} can be reformulated as
{\mathindent=1mm
\begin{align}\label{J:IA2:modified}
  \mathfrak{J}  =  \sum\limits_{k\in\K}\sum\limits_{\ell\in\K\setminus k}
  \mathrm{tr}\left\{\pmb{P}_\ell^H\pmb{\Delta}_\ell^H\pmb{H}_{k,\ell}^H(\pmb{I}- \pmb{U}_k\pmb{U}_k^H)\pmb{H}_{k,\ell}\pmb{\Delta}_\ell\pmb{P}_\ell\right\}.
\end{align}}
By holding $\{\pmb{U}_k\}_{k=1}^{K}$ fixed, we select the power allocation matrix $\pmb{P}_\ell$  as the optimal solution of the following optimization problem
{\mathindent=1mm
\begin{subequations}\label{OptProP}
\begin{align}
 \underset{\pmb{P}_\ell}{\text{min}}\,\,&
 \mathrm{tr}\{\pmb{P}_\ell^H(\sum_{k\in\K\setminus \ell}\pmb{\Delta}_\ell^H\pmb{H}_{k,\ell}^H(\pmb{I}- \pmb{U}_k\pmb{U}_k^H)\pmb{H}_{k,\ell}\pmb{\Delta}_\ell)\pmb{P}_\ell\}\\
 \text{s.t.} \,\,& \pmb{P}_\ell^H\pmb{P}_\ell=  \frac{P_t}{d}\pmb{I},
\end{align}
\end{subequations}}
which results in
{\mathindent=1mm
\begin{align}\label{Popt}
  \pmb{P}_\ell = \sqrt{\frac{P_{t}}{d}}\zeta_{\min}^d \{\sum\limits_{k\in\K\setminus \ell}\pmb{\Delta}_\ell^H\pmb{H}_{k,\ell}^H(\pmb{I}- \pmb{U}_k\pmb{U}_k^H)\pmb{H}_{k,\ell}\pmb{\Delta}_\ell\}.
\end{align}}

\emph{Receiver interference subspace selection}:
Now, $\left\{\pmb{P}_\ell\right\}_{\ell=1}^K$ is hold fixed, the orthonormal basis matrix $\pmb{U}_k$ at the legitimate Rx $k$ is selected for each Rx $k$ to ensure that most of interference falls into the interference subspace $\mathcal{U}_k$ spanned by $\pmb{U}_k$. In order to minimize the interference signals in $\mathfrak{J}$ in \eqref{J:IA2}, the matrix distance minimization problem is equivalent to
{\mathindent=1mm
\begin{subequations}\label{OptProUk}
\begin{align}
 \underset{\pmb{U}_k}{\text{max}} & \,\,\,\,\,\,\, \mathrm{tr}\{\pmb{U}_k^H(\sum\limits_{\ell\in\K\setminus k}\pmb{H}_{k,\ell}\pmb{\Delta}_\ell\pmb{P}_\ell\pmb{P}_\ell^H\pmb{\Delta}_\ell^H\pmb{H}_{k,\ell}^H)\pmb{U}_k\} \\
 \text{s.t.} &\,\,\,\,\,\,\, \pmb{U}_k^H\pmb{U}_k = \pmb{I}.
\end{align}
\end{subequations}}
Hence, the solution of \eqref{OptProUk} can be obtained by \cite{Lutke97}, given as below
\begin{align}\label{UkIA2}
  \pmb{U}_k = \zeta_{\max}^{N-d}\{\sum\limits_{\ell\in\K\setminus k}\pmb{H}_{k,\ell}\pmb{\Delta}_\ell\pmb{P}_\ell\pmb{P}_\ell^H\pmb{\Delta}_\ell^H\pmb{H}_{k,\ell}^H\}.
\end{align}
The step-by-step iterative algorithm for the ZFWS IA scheme is shown in Algorithm \ref{alg3}.
\begin{algorithm}[ht]
\caption{: Proposed ZFWS IA Algorithm for Secure Multiuser MIMO Systems}
\begin{algorithmic}[1]\label{alg3}
\STATE Inputs: $d,\pmb{H}_{k,\ell}$, $\forall k,\ell \in \mathcal{K}$, $\kappa=0$, $\kappa_{\max }$ where $\kappa$ is the iteration index;
\STATE Evaluate $\pmb{\Delta}_\ell$  by choosing the last $d$ columns of $\pmb{\Xi}_\ell$ in \eqref{HE:SVD}, $\forall \ell \in \mathcal{K}$;
\STATE Initial variables: Set $\{\pmb{U}_k^{(0)}\}_{k = 1}^{K}$ are identity matrices; then select $\{\pmb{P}_k^{(0)}\}_{k = 1}^{K}$ from \eqref{Popt};
\STATE Evaluate the objective function $\mathfrak{J}(\{\pmb{P}_k^{(0)}\}_{k=1}^K, \{\pmb{U}_k^{(0)}\}_{k=1}^{K})$ from \eqref{J:IA2:modified};
\WHILE{$\kappa < \kappa_{\max}$}
\STATE For fixed $\{\pmb{P}_k^{(\kappa)}\}_{k = 1}^{K}$, select $\{\pmb{U}_k^{(\kappa+1)}\}_{k = 1}^K$ from \eqref{UkIA2};
\STATE For fixed $\{\pmb{U}_k^{(\kappa+1)}\}_{k = 1}^K$, select $\{\pmb{P}_k^{(\kappa+1)}\}_{k = 1}^{K}$ from \eqref{Popt};
\STATE Evaluate the objective function $\mathfrak{J}(\{\pmb{P}_k^{(\kappa+1)}\}_{k=1}^K, \{\pmb{U}_k^{(\kappa+1)}\}_{k=1}^{K})$ from \eqref{J:IA2:modified};
\STATE $\kappa=\kappa+1$;
\STATE Repeat steps 6-9 until convergence.
\ENDWHILE
\end{algorithmic}
\end{algorithm}

\emph{Convergence analysis for the ZFWS IA design}:
The convergence of the second IA scheme is presented in the following theorem.
\begin{theorem}\label{thrmalg3}
The proposed ZFWS IA algorithm for secure multiuser MIMO communication networks converges monotonically.
\end{theorem}
\begin{proof}
The proof is similar to that of Theorem \ref{thrmalg2} and, thus, omitted.
\end{proof}

\emph{Feasible condition for the ZFWS IA scheme}:
Similar to the WSLM IA scheme, the feasibility of the ZFWS IA scheme relies on the number of the variables and that of the equations. The  properness of the second IA design is stated in the following theorem.
\begin{theorem}\label{thrmalg4}
The feasible condition of the proposed ZFWS IA algorithm for secure multiuser MIMO communication networks is $M-d\geq N_e$ and $N\geq Kd$.
\end{theorem}
\begin{proof}
First, $M-d\geq N_e$ is the necessary condition for the existence of $\pmb{F}_\ell$ satisfying condition \eqref{IA:zero-leakage-constrain}. Exploiting the structure of the precoding matrix $\pmb{F}_\ell = \pmb{\Delta}_\ell\pmb{P}_\ell$  and removing superfluous variables \cite{yetis09}, the number of variables and equations in the IA  condition \eqref{ConventionalIA:SecrecyConditions2} are $N_v = K(N-d)d$  and $N_{eq} = K(K-1)d^2$, respectively.
The feasible condition is given by $N_v\geq N_{eq}$ which results in $N\geq Kd$.
\end{proof}

\begin{remark}
We now compare the constraints on the number of antennas required for proper systems of the proposed IA schemes.
Since the feasible condition of the ZFWS IA design is $N \geq Kd$,
the feasible condition of the WSLM IA design \eqref{feasibleIA1condition} can be rewritten as
{\mathindent=1mm
\begin{align}\label{feasibleIA1conditon:compared1}
  \frac{KM-d}{K-1} \leq \frac{K(M+N)-(K^2+1)d}{K-1},\,\,\,\forall N \geq Kd.
\end{align}}
It is noted that
\begin{equation}\label{feasibleIA1conditon:compared2}
 \frac{KM-d}{K-1} > M-d; \,\,\,\forall K \geq 2.
\end{equation}
Hence, from \eqref{feasibleIA1conditon:compared1} and \eqref{feasibleIA1conditon:compared2}, we have
\begin{align}\label{feasibleIA1conditon:compared3}
\nonumber
\frac{K(M+N)-(K^2+1)d}{K-1} &> M-d, \,\,\,
\\
\forall K \geq 2, \forall N \geq Kd.
\end{align}
It means that the upper bound of $N_e$ in the WSLM IA design is always higher than that of the ZFWS IA design.
\end{remark}
\section{Illustrative Results}
\label{sec:Results}
In this section, we evaluate the SSR performance of our two proposed IA designs through several numerical results, in comparison with the conventional IA algorithm \cite{peters09}. In simulations, noise variances are normalized $\sigma_k^2=\sigma^2=1$. The Rayleigh fading channel coefficients are generated from the complex Gaussian distribution ${\mathcal{CN}}(0,1)$. We define signal-to-noise-ratio $\text{SNR}=\frac{P_{t}}{\sigma^2}$. All the numerical results are averaged over the $200$ channel realizations.

First, we investigate the convergence characteristic of the two proposed IA algorithms.  We run the simulation for ${\left( 9 \times 9,6,3 \right)^3}$ and ${\left( 6 \times 6,4,2 \right)^3}$ systems with a random channel realization. Note that both systems are satisfied the feasible conditions in Theorems \ref{thrmalg2} and \ref{thrmalg4}. The evolution of the cost functions (SSR) over iterations is illustrated in Fig. \ref{fig:Convergence}. As can been seen from Fig. \ref{fig:Convergence}, the cost functions reduce monotonically over iterations and approach to zero. The second proposed IA design converges very fast in just few iterations while the first proposed IA design takes around $25$ iterations such that the objective value converges to $10^{-9}$.

\begin{figure}[htb!]
\includegraphics[keepaspectratio,width=3.6in]{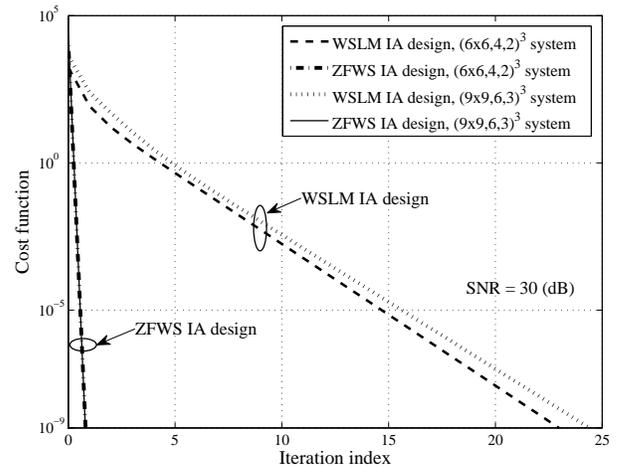}
\caption{The convergence behavior of the proposed IA algorithms.}
\label{fig:Convergence}
\end{figure}
\vspace{+1cm}

Now, we compare the SSR of our two proposed IA algorithms with that of the conventional IA method \cite{peters09} which does not involve the secrecy constraints. First, we consider $4$ scenarios of ${\left( 9 \times 9,6,3 \right)^3}$ , ${\left( 9 \times 9,9,3 \right)^3}$, ${\left( 15 \times 15,9,3 \right)^3}$ and ${\left( 15 \times 15,18,3 \right)^3}$  systems.
It should be noted that the WSLM IA feasible condition is satisfied in all of these systems but only the ${\left( 9 \times 9,6,3 \right)^3}$ and ${\left( 15 \times 15,9,3 \right)^3}$ systems satisfy the feasible conditions for the ZFWS IA scheme while the other systems do not. It has been revealed from Fig. \ref{fig:SSR1} and Fig. \ref{fig:SSR2} that for the ${\left( 9 \times 9,6,3 \right)^3}$ and ${\left( 15 \times 15,9,3 \right)^3}$ systems, the ZFWS IA scheme outperforms the WSLM IA one.

\begin{figure}[htb!]
\includegraphics[keepaspectratio,width=3.6in]{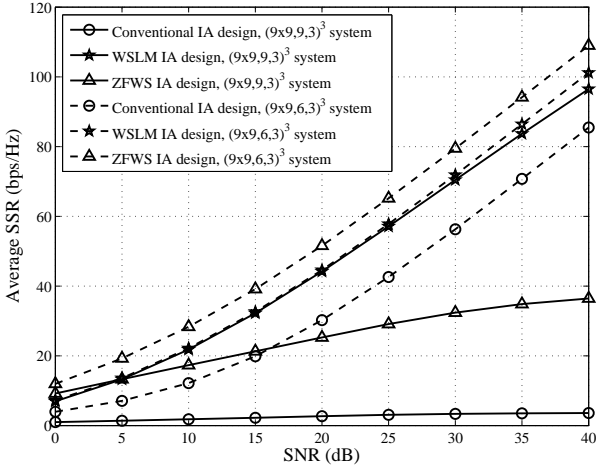}
\caption{The average SSR versus SNR for ${\left( 9 \times 9,6,3 \right)^3}$ and ${\left( 9 \times 9,9,3 \right)^3}$ systems.}
\label{fig:SSR1}
\end{figure}
\vspace{+1cm}

\begin{figure}[htb!]
\includegraphics[keepaspectratio,width=3.6in]{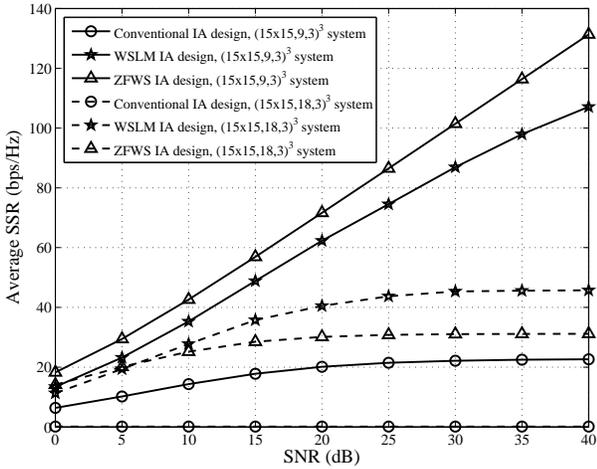}
\caption{The average SSR versus SNR for ${\left( 15 \times 15,9,3 \right)^3}$ and ${\left( 15 \times 15,18,3 \right)^3}$ systems.}
\label{fig:SSR2}
\end{figure}
\vspace{+1cm}

However, note that the constraint on the number of antennas at the eavesdropper in the ZFWS IA method is more restrictive than that of WSLM IA scheme. Thus, the ${\left( 9 \times 9,9,3 \right)^3}$ and ${\left( 15 \times 15,18,3 \right)^3}$ systems do not meet the feasible condition of the ZFWS IA scheme while still satisfying the feasible condition of the WSLM IA scheme. In two these configurations, the SSRs of the WSLM IA scheme are higher than the SSRs of the ZFWS IA one. It is also important to observe from Fig. \ref{fig:SSR1} and Fig. \ref{fig:SSR2} that both proposed IA designs significantly improve the SSR in comparison with the conventional IA design.

Finally, we investigate the SSR improvement of two proposed IA design in comparison with the conventional IA design versus various number of antennas at the eavesdropper ($N_e$). We define the average SSR improvement of each proposed IA design in comparison with the conventional IA design as follows
{\mathindent=1mm
\begin{align}\label{SSRImprovement}
  \mathcal{R}_S^{Imp} = \frac{1}{N_R}\sum\limits_{n=1}^{N_R}\left[\mathcal{R}_S^{IAproposed} - \mathcal{R}_S^{IAconventional}\right]
\end{align}}
where $N_R$ is the number of channel realizations, $\mathcal{R}_S^{IAproposed}$ and $\mathcal{R}_S^{IAconventional}$ are the SSRs of our proposed IA designs and the conventional IA design, respectively. Fig. \ref{fig:SSRImprovement} illustrates the average SSR improvements of two proposed IA designs in the ${\left( 9 \times 9,N_e,3 \right)^3}$ and ${\left( 15 \times 15,N_e,3 \right)^3}$ systems with $SNR=30$ dB. It can be observed that the SSR improvement is not significant as $N_e$ is too low or too high. The reasons for this, are that, when $N_e$ is high, the feasible conditions for IA schemes are not satisfied and, then, the secure communication is not feasible. On the other hand, when $N_e$ is low, the rate of wiretap channels is negligible and, thus, the SSR improvement is not significant. Additionally, Fig. \ref{fig:SSRImprovement} also reveals that the ZFWS IA scheme offers the better SSR improvement than the WSLM IA method as the feasible condition of the ZFWS IA is satisfied.  These results give an instruction for selecting an appropriate IA approach for given system parameters to gain the best secrecy.

\begin{figure}[htb!]
\includegraphics[keepaspectratio,width=3.6in]{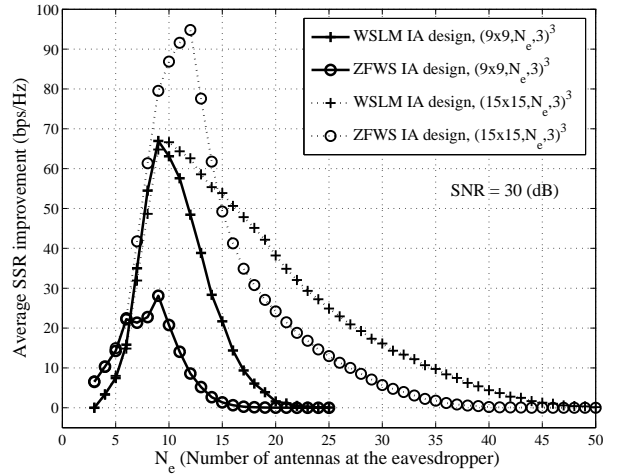}
\caption{The average SSR improvement versus the number of antennas at the eavesdropper for ${\left( 9 \times 9,N_e,3 \right)^3}$ and ${\left( 15 \times 15,N_e,3 \right)^3}$ systems.}
\label{fig:SSRImprovement}
\end{figure}
\vspace{+0.5cm}

\section{Conclusion}
\label{sec:Conclusion}
We have designed two IA schemes to enhance the secrecy of multiuser MIMO communication systems in flat-fading channels. Our two IA approaches have been designed to minimize the IMLI at all Rxs and the WSs at the eavesdropper by properly choosing the transmit precoding and receiving subspace matrices at the legitimate users. Our analysis provides key insights into the two proposed IA designs over the proof of convergence and feasible conditions. Although each proposed IA approach has its own advantages and disadvantages based on feasible conditions over the number of antennas at the eavesdropper, both proposed IA methods demonstrate significant SSR improvements in comparison with the conventional IA method.
\\
\\
\\

\section*{Acknowledgments}
This research is funded by Vietnam National Foundation for Science and Technology Development (NAFOSTED) under grant number 102.04-2013.46.

\balance
\bibliographystyle{IEEEtran}
\bibliography{IASecrecy2015}
\end{document}